\documentclass[12pt]{article}
\usepackage{amsmath,amstext,amsfonts,amsthm}
\usepackage{amssymb,hyperref,bbm,verbatim}
\usepackage{marvosym,color}
\usepackage[T1]{fontenc}

\newtheorem{thm}{Theorem}
\newtheorem*{thm*}{Theorem}
\newtheorem{lemma}{Lemma}[section]

\newtheorem{rmk}[lemma]{Remark}

\newtheorem{claim}[lemma]{Claim}

\def\R{\mathbb{R}}
\def\C{\mathbb{C}}

\def\be{\begin{equation} }
\def\ee{\end{equation} }
\DeclareMathOperator{\Tr}{\mathrm{Tr}}

\def \E#1{\mathbb {E} \left[ #1 \right] }    
\def \Pr#1{\mathbb {P} \left\{ #1 \right\} } 
\def \ket #1 {| #1 \rangle}
\def \bra #1 {\langle  #1 |}
\def\vp {\varphi}

\DeclareMathOperator{\F}{F}
\DeclareMathOperator{\op}{op}

\numberwithin{equation}{section}

\begin{document}

\title{Matrix regularizing effects of  Gaussian perturbations}

\author{
Michael Aizenman\textsuperscript{1}\quad
Ron Peled\textsuperscript{2} \quad
Jeffrey Schenker\textsuperscript{3} \\
Mira Shamis\textsuperscript{4,5} \quad
 Sasha Sodin\textsuperscript{4,2} }
\footnotetext[1]{Departments of Mathematics and Physics, Princeton
University, Princeton, NJ 08544, USA.
}
\footnotetext[2]{School of Mathematical
Sciences, Tel Aviv University, Tel Aviv, 69978, Israel.
}
\footnotetext[3]{Department of Mathematics, Michigan State
University, East Lansing, MI 48824, USA.
}
\footnotetext[4]{School of Mathematical
Sciences, Queen Mary University of London, London E1~4NS, United Kingdom.}
\footnotetext[5]{Department of Mathematics, The
Weizmann Institute of Science, Rehovot 7610001, Israel.
}
\maketitle

\begin{abstract}
The addition of noise has a regularizing effect on Hermitian matrices.   This effect is studied here for  $H=A+V$, where $A$ is the base matrix and $V$ is sampled from the GOE or the GUE random matrix ensembles.   
We bound the mean number of eigenvalues of $H$ in an interval, and present 
tail bounds for the distribution of the Frobenius and operator norms of $H^{-1}$ and for the distribution of the norm of $H^{-1}$ applied to a fixed vector.  The bounds are uniform in  $A$ and exceed the actual suprema by no more than multiplicative constants. 
The probability of multiple eigenvalues in an interval is also estimated.  
\end{abstract}

\section{Introduction}

It is often the case that disorder has a regularizing 
 effect
on the spectrum of an Hermitian matrix.
Recall the Wegner
estimate \cite{W3}, which express the regularizing 
 effect
of diagonal disorder, and which is central in the spectral
analysis of random operators.    The  estimate was formulated for  matrices of the form
\be \label{eq:V_D}
A +  V^{\text{diag}} \, ,
\ee
where $A$ is Hermitian and $V^{\text{diag}}$ is diagonal with entries independently sampled from a bounded probability density $\rho$ on $\R$. For such $N\times N$ matrices, one has uniformly in $A$:
\begin{equation}\label{weg_dos}
\E{\mbox{\#\{eigenvalues of $(A + V^{\text{diag}})$ in $I$\}}}  \leq   \| \rho \|_\infty \,  N \, |I|\quad \mbox{for  any interval $I\subset \R$~,}
\end{equation}
where $\|\rho\|_\infty$ is the essential supremum of $\rho$,
and $|I|$ the Lebesgue measure of $I$. The following related estimate is also valid:
\begin{equation}
\Pr{\left|  (A +  V^{\text{diag}})^{-1}_{j\, j} \right|  > t}  \leq  \frac{\| \rho \|_\infty }{t} \quad \qquad \mbox{for all $j = 1,...,N$~.}  \label{weg_res}
\end{equation}

Presented here are somewhat analogous bounds for  matrices of the form
\be
H \ = \ A^{\text{sym}} + V^{\text{GOE}} \mbox{\qquad or  \qquad $H \ = \ A^{\text{Herm}} + V^{\text{GUE}} $ }\, ,
\ee
the first case concerning a real symmetric  base matrix $ A^{\text{sym}}$ perturbed by a random matrix $V^{\text{GOE}} $ sampled from the Gaussian Orthogonal Ensemble,
and  the second case concerning an Hermitian base matrix  perturbed by a random matrix
sampled from the Gaussian Unitary Ensemble.  The superscripts, which are displayed here for clarity, will often be omitted.

The invertibility properties of $H=A+V$ are
quantified in several ways: i) tail bounds for the distribution of the norm of
$H^{-1}\vp $ when $\vp$ is a fixed vector, ii)  corresponding bounds for the Frobenius and operator norms of
$H^{-1}$, iii) a bound on the expected number of eigenvalues
of $H$ in an interval.   
The bounds are uniform in $A$ and exceed the actual suprema by no more than multiplicative constants, as can be seen by considering the case $A=0$ 
(cf.\ Section~\ref{sec:disc}).\\ 

To state the results precisely we first recall the definitions of
the invariant ensembles.   These consist of Hermitian matrices of the form
\begin{equation}\label{X}
  V=\frac{X + X^*}{\sqrt{2N}},
\end{equation}
where $X$ is an $N\times N$ matrix with independent  standard real Gaussian entries in case of GOE, or independent standard complex Gaussian entries in case of GUE,
and the asterisk  indicates Hermitian conjugation.
In both cases the probability distribution of $V$ is of density proportional to
\[ \exp\left\{ - \frac{\beta N}{4} \mathrm{tr}\, V^2\right\} \]
with respect to the Lebesgue measure on matrices of the corresponding symmetry: real symmetric (GOE,  with $\beta = 1$) or complex Hermitian (GUE, with $\beta = 2$).   The distributions are invariant under conjugation by the corresponding class of unitary matrices
(cf.  \cite{AGZ,Mehta,PShch},  where various aspects of the invariant Gaussian ensembles are discussed).  \\

Throughout we write $\|\vp\|$ for the Euclidean norm of a vector $\vp$,
and for a matrix $R$ write $\|R\|_{\F} = \sqrt{\Tr\,R R^*}$  for the Frobenius (Hilbert-Schmidt) norm and
$\|R\|_{\op} = \max_{\vp \neq 0} \|R\vp\|/\|\vp\|$ for the operator norm.   The following  pair of theorems states our main results.

\begin{thm}\label{thm:main}
\hfill
\begin{tabbing}
\emph{If either:}  \= $A$ is
  an $N\times N$ real symmetric \= matrix, \= $\vp\in\R^N$, and \= $V$ is sampled from  GOE,  \\
\emph{or:} \> $A$ is
  an $N\times N$ Hermitian \> matrix,  \> $\vp\in\C^N$, and \> $V$ is sampled from  GUE,  \\
\end{tabbing}
\vspace{-.5cm}
then  the following bounds apply to the matrix  $H:=A+V$, with a constant
$C<\infty $ which is uniform in  $N$, $A$, and $\vp$:
\begin{enumerate}
  \item (Fixed vector)  for all $t\geq 1$,
  \begin{equation}\label{eq:vec}
    \mathbb{P} \left\{ \| H^{-1} \vp \| \geq t \sqrt{N} \|\vp\| \right\} \leq \frac{C}{t} \, ,
  \end{equation}
  \item (Frobenius and operator norms)  for all $t\geq 1$,
  \begin{equation}\label{eq:frob}
\mathbb{P} \left\{ \| H^{-1} \|_{\op} \geq tN \right\}\  \leq
\     \mathbb{P} \left\{ \| H^{-1} \|_{\F} \geq t N \right\} \  \leq \  \frac{C}{t}\,,
  \end{equation}
  \item (Mean density of states) for any interval
  $I\subset\R$,
  \begin{equation}\label{eq:dos}
 \E { \# \left\{\mbox {\! eigenvalues of $H$ in $I$\! }\right\} }  \  \leq \   C   N|I|~.
  \end{equation}
\end{enumerate}
\end{thm}

The key for the three statements listed in Theorem~\ref{thm:main} is the single-vector tail estimate
(\ref{eq:vec}). In our approach the two other bounds are concisely derived from it. The main technical step in
the proof of (\ref{eq:vec}) is (\ref{eq:lmain}) of Lemma~\ref{lemma:main}. 
Estimates of similar nature (concentration bounds on quadratic forms)
have also played a role in the work of Maltsev
and Schlein \cite{MSch1,MSch2} and Vershynin \cite{V}.  
The  lemma is proved below by a  Fourier-analytic method.

Because of the similarity between (\ref{eq:dos}) and (\ref{weg_dos}) of \cite{W3}, the former bound may be referred to as a Wegner-type estimate (though the similarity of the bounds does not extend to their derivations).
For GUE perturbations of Hermitian matrices a  bound of the form (\ref{eq:dos}) on the density of states, from
which (\ref{eq:frob}) can be deduced for that case, was  also recently proved by Pchelin \cite{Pch}, building on
the work of Shcherbina \cite{Shch}.  

For the next statement, we denote, for  Borel sets  
$B \subset \R$ and Hermitian matrices $H$:
\begin{equation*}
  \mathcal{N} (B) = \mathcal{N}(B; H) := \text{number of eigenvalues of $H$ in $B$}.
\end{equation*}
\begin{thm}\label{thm:Minami_like}
  Let $H = A + V$ be as in Theorem~\ref{thm:main}. Then there is a
  constant $C<\infty$, uniform in $A$ and $N$, such that for every $1\le k\le N$ and every interval $I \subset \mathbb{R}$,
  \begin{equation}\label{eq:Minami_probability}
    \mathbb{P}\left\{\mathcal{N}(I)\ge k\right\} \leq  \frac{\left(C\,  |I| \, N\right)^k}{k!}.
  \end{equation}
 Moreover, for any $k$-tuple of intervals
   $I_1, \cdots, I_k \subset\R$,
  \begin{equation}\label{eq:Minami_expectation}
    \mathbb{E}\left[\mathcal{N} (I_1)(\mathcal{N} (I_2)-1)_+\cdots(\mathcal{N} (I_k)-k+1)_+\right]\le
    \prod_{j=1}^k (C\, |I_j|\, N).
  \end{equation}
\end{thm}
Continuing the comparison with bounds which are known for operators with random   potential, the case $k = 2$ of (\ref{eq:Minami_probability}) is reminiscent of the Minami bound for 
matrices $A+V^{\operatorname{diag}}$
with diagonal disorder, as in (\ref{eq:V_D}), for which it was established in \cite{Minami} that
\begin{equation}  \label{eq:Minami}
  \mathbb{P}\left\{\text{$(A + V^{\text{diag}})$ has at least $2$ eigenvalues in $I$}\right\} \leq  C \left(\| \rho \|_\infty \,  |I| \, N\right)^2 ~.
\end{equation}
This estimate was instrumental in Minami's proof of Poisson local eigenvalue statistics for
random Schr\"odinger operators in $\mathbb{Z}^d$ throughout the regime of Anderson localization.   
Extensions to $k > 2$ 
were subsequently presented  by Bellissard, Hislop, and Stolz \cite{BHS}, Graf and Vaghi \cite{GV},
and  Combes, Germinet, and Klein \cite{CGK}. In particular, our derivation of Theorem~\ref{thm:Minami_like} has benefitted from  the strategy of \cite{CGK}.

\paragraph{Applications}

The above bounds are useful for a number of problems in the theory
of random operators, particularly,  pertaining to random band and Wegner-type operators, some of which  will be discussed in \cite{PSSS}.
The estimate \eqref{eq:vec} plays a key role in the proof
of  localization at strong disorder for the Wegner $N$-orbital model, and some of its
variants, with conjecturally sharp dependence of the localization
threshold on the number of orbitals.  The bound \eqref{eq:dos}
enables density of state estimates for a class of models including the
Wegner orbital model and Gaussian band matrices.  Theorem~\ref{thm:Minami_like} is used to prove convergence
of the local eigenvalue statistics to the Poisson process in the regime of localization. In such applications
the sharp dependence of the above bounds on $N$ and $t$ is of value.

\medskip

The bounds discussed here are  of relevance also from other perspectives.
Effects on the spectrum of the addition of a symmetric random matrix
has been studied in light of applications in numerical analysis by  
Sankar, Spielman and Teng \cite{SST} (for the case of Gaussian random
matrices) and by 
Vershynin \cite{V} and Farrell and Vershynin \cite{FV} (for more general
distributions).
The addition of GOE/GUE and its infinite volume limit 
were studied by Dyson \cite{Dyson} in the context of stochastic evolution,
by Pastur \cite{P} in the framework of the limiting eigenvalue
distribution for deformed Wigner ensembles. The regularisation
effect in the infinite volume limit was considered in \cite{BV}
in the language of free convolution of Voiculescu
\cite{Voic}.

\paragraph{Relation with previous results}

$ $  In presenting some of the related previous results we
shall invoke the notion of density of states, and the following notation. 
For an $N\times N$ random matrix $H$,  the normalized average
$\nu(\cdot ;H) := N^{-1} \mathbb{E} \left[\mathcal{N}(\cdot; H) \right]$ (or just $\nu$)
is referred to as the
density of states  (DOS) measure.   When this measure is absolutely continuous,
i.e. of the form  $\nu(d{\mathcal E}) = \rho({\mathcal E}) \, d{\mathcal E}$,
 its  Radon density  $\rho({\mathcal E})$    is called the density of states function.
 In this notation, the bound~\eqref{eq:dos} asserts that the DOS
 measure $\nu(\cdot;H)$ of $H = A+V $ is absolutely continuous, and its
density $\rho({\mathcal E}; H)$ is bounded by a constant independent of $N$ and $A$.\\

 While the results presented here focus on  bounds which hold uniformly in the base matrix $A$,
related questions have been studied for sequences  $A_N$ of deterministic Hermitian matrices
of increasing size for which the density of state measures  $\nu_N$ converge
weakly to a limiting measure $\nu_\infty(d\mathcal{E})$.    Pastur~\cite{P} has shown that in such situations  the  perturbed operators  $A_N  + V_N^\text{GOE/GUE} $
(and more generally $A_N + V_N^\text{Wig}$, see below) have densities of states which converge weakly to a limit which can be determined from  $\nu_\infty$,
and which is absolutely continuous of density satisfying $\rho_\infty({\mathcal E}) \leq \pi^{-1}$ (c.f. the monograph \cite{PShch}).

There are also several results which rely on the Harish--Chandra formul{\ae} \cite{HC,BH1,BH2}, and thus apply
to GUE but not GOE perturbation. For the case
that $A_N$ are uniformly norm-bounded and the perturbation
is GUE it is a by-product
of the study of local eigenvalue statistics by T.~Shcherbina \cite{Shch,Shch2} 
that the Pastur law also holds in total variation distance. Thus one can conclude that
\[ \sup_{\| A_N\| \leq K} \sup_{\mathcal E} \rho({\mathcal E}; A_N +V_N^\text{GUE} ) \leq \frac{1}{\pi}+ o(1), \quad N \to \infty,\]
This
bound is similar to the GUE case of \eqref{eq:dos}, but it requires the deformation
to be bounded. For the case of possibly unbounded Hermitian matrix
perturbed by the GUE, Pchelin proved that
\[  \sup_N \sup_{A_N} \sup_{\mathcal E} \rho({\mathcal E}; A_N +V_N^\text{GUE} ) < \infty~, \]
i.e.\ our bound (\ref{eq:dos}) on the mean density of
states; his argument builds on  \cite{Shch}.

The above question was considered also in the more general setting obtained by replacing
GOE/GUE  by  Wigner matrices $V_N^\text{Wig}$, for which the entries above the main diagonal
are iid though not necessarily Gaussian.   Vershynin \cite{V} showed that in such case
\begin{equation}\label{eq:ver}
\sup_{\| A_N\| \leq K} \mathbb{P} \left\{ \| (A_N + V_N^\text{Wig})^{-1} \|_\text{op} \geq t N \right\} \leq \frac{C_K}{t^{1/9}} +
2 \exp(-N^{c_K})
\end{equation}
with constants $C_K, c_k>0$ depending only on $K$.  Vershynin's result
holds under very mild assumptions on the matrix entries; an
inspection of the proof shows that if the entries are themselves
regular (for example, have density bounded by $C\sqrt{N}$), the estimate holds
without the term $2 \exp(-N^{c_K})$. We also mention
that Nguyen \cite{Nguyen} showed
that for any $K > 0$ and
$b>0$  there exists $a>0$ so that
\begin{equation}\label{eq:Ng}
\sup_{\| A_N\| \leq N^K} \mathbb{P} \left\{ \| (A_N + V_N^\text{Wig})^{-1} \|_\text{op} \geq N^a \right\} \leq N^{-b}~.
\end{equation}
Upper bounds on the probability of two close eigenvalues were proved by Nguyen,
Tao and Vu \cite{NTV}.

Recently, universality of local eigenvalue statistics for deformed
Wigner ensembles was studied by O'Rourke and Vu \cite{OV},
Knowles and Yin \cite[Section 12]{KY} and Lee, Schnelli, Stetler, and Yau \cite{LSSY}.

\medskip\noindent

Among the results pertaining to $A_N=0$, that is
 concerning the density of state  of the Wigner matrices
without these being used as deformations of
 a base matrix, we mention only a few most relevant
to the current discussion.

One of the forms of the Wigner law asserts that if $V_N$ is sampled
from a Wigner ensemble of dimension $N\times N$ then
\begin{equation}\label{eq:wig}
\rho({\mathcal E}; V_N^{\text{Wig}}) \to \frac{1}{2\pi} \sqrt{(4-{\mathcal E}^2)_+}, \quad
N \to \infty
\end{equation}
in the weak sense \cite{AGZ,PShch}.  In the special cases of GOE and GUE, this
may be strengthened to uniform convergence \cite{Mehta}, yielding
\begin{equation}\label{eq:wigcor}
\sup_{\mathcal E} \rho({\mathcal E}, V_N^{\text{GOE/GUE}}) \leq \frac{1}{\pi} + o(1), \quad N
\to \infty.
\end{equation}
This implies a bound  similar  to  \eqref{eq:dos} but for $A_N=0$.\\

Maltsev and Schlein \cite{MSch2} proved that the Wigner law (\ref{eq:wig}) holds in
the topology of uniform convergence in $[-2+\delta, 2-\delta]$ (for
an arbitrary $\delta > 0$) for a class of Wigner matrices
the entries of which obey certain regularity
assumptions. Their results imply that (\ref{eq:wigcor})
with the restriction $|{\mathcal E}|<2-\delta$
holds for this class of Wigner matrices. The paper \cite{MSch2} builds on earlier work by
Erd\H{o}s, Schlein, and Yau \cite{ESY} and Maltsev and Schlein
\cite{MSch1}, where it was shown that there exists an absolute
constant $C>0$ for which
\[ \mathbb{P} \left\{ \| (V_N^\text{Wig} - {\mathcal E})^{-1} \|_\text{op} \geq t N \right\} \leq \frac{C}{t},  \quad t\ge
1.\]

\section{Fixed vector bound}  \label{lem:fixed_vec}

In general terms, the  Wegner bound concerns  the inverse of a quantity which fluctuates due to the presence of random terms in $H$.   For the  bound \eqref{weg_res}  it suffices to focus on the fluctuations  resulting from the randomness in the single term
$V^{\text{diag}}_{jj}$.  However,  for the case considered here the contribution of any single diagonal term is too small  (by a factor of $\sqrt N$) for the claimed result.    Instead, our proof of the bound \eqref{eq:vec}, for a given $N\times N$ matrix $A$, and a given  vector $\vp$,  will focus on the fluctuations in $\| H^{-1} \vp \|$  due to the $N$ random variables which determine $V \vp$.     We start by reducing the claim to a technical estimate whose proof will be given separately, in Section~\ref{sec:tech}.   \\

\noindent {\bf The real (GOE) case.}   As the distribution of $V$ is
invariant under  orthogonal conjugations, and we aim at results which hold uniformly in $A$,  we may assume without loss of generality that $\vp = e_1$, the first vector of the standard basis in $\mathbb{R}^N$.  The
matrix $V$ has the form
\begin{equation}\label{eq:defbl}
V = \frac{1}{\sqrt{N}} \left(  \begin{array}{cc}
\sqrt{2} g_0   & g^t \\
g & W
\end{array}\right),\end{equation}
where $g_0\in\R$, $g\in\R^{N-1}$ and $W$ is an $(N-1)\times(N-1)$ symmetric
matrix,  $g_0$, $g$ and $W$ are independent,  and $g_0$ is a standard real Gaussian and
$g$ is a standard real Gaussian vector (i.e., with independent entries having the standard real Gaussian distribution).
Thus we may write
\begin{equation}\label{eq:A + V}
A + V  = \frac{1}{\sqrt{N}} \left(  \begin{array}{cc}
\sqrt{2} g_0 +a    & (g+b)^t \\
g+b & W + D
\end{array}\right),\end{equation}
for deterministic $a\in\R$, $b\in\R^{N-1}$ and $D$ an $(N-1)\times
(N-1)$ symmetric matrix.
Set
\begin{equation*}
  Q := (W + D)^{-1}.
\end{equation*}
Inverting using the
Schur--Banachiewicz formul{\ae}, we obtain
\begin{equation}
\frac{1}{\sqrt{N}} (A + V)^{-1} e_1 =
\frac{1}{\sqrt{2} g_0 + a - (g+b)^t Q (g+b)} \left( \begin{array}{c}
1 \\ -Q(g+b)
\end{array} \right).
\end{equation}
Therefore
\begin{multline}\label{eq:deformed_GOE_inverse_formula}
  \frac{1}{\sqrt{N}} \| (A + V)^{-1} e_1 \|
 = \frac{\sqrt{1 + \| Q (g+b)\|^2}}{|\sqrt{2} g_0 + a - (g+b)^t Q (g+b)|}\\
 \leq \frac{1 }{|\sqrt{2} g_0 + a - (g+b)^t Q (g+b)|} \ + \
\frac{\| Q (g+b)\|}{|\sqrt{2} g_0 + a - (g+b)^t Q (g+b)|}.
\end{multline}
For any deterministic $d$, and any  $t>0$,
\begin{equation*}
\mathbb{P} \left\{ \left| \frac{1}{\sqrt{2}g_0 + d} \right| \geq t
\right\} \leq \frac{1}{\sqrt{\pi} t},
\end{equation*}
therefore, first conditioning on $g$ and $Q$, one may conclude that
\begin{equation}\label{eq:vec1}
\mathbb{P} \left\{ \frac{1}{|\sqrt{2} g_0 + a - (g+b)^t Q (g+b)|}
\geq \frac{t}{2} \right\} \leq \frac{2}{\sqrt{\pi} t}.
\end{equation}
Combining \eqref{eq:deformed_GOE_inverse_formula} with \eqref{eq:vec1} one  arrives at the  key bound
\begin{equation} \label{eq:key}
  \mathbb{P} \left\{ \| (A + V)^{-1} e_1 \| \geq t \sqrt{N} \right\} \leq \frac{2}{\sqrt{\pi}t} + \mathbb{P} \left\{ \frac{\| Q (g+b)\|}{|\sqrt{2} g_0 + a - (g+b)^t Q
(g+b)|} \geq \frac{t}{2} \right\}.
\end{equation}
For the second term we have the following estimate, whose proof is deferred to Section~\ref{sec:tech}.

\begin{lemma}\label{lemma:main}
Let $Q$ be a (non-random) non-zero real symmetric matrix, and let $g$ be a standard real Gaussian vector of the same
dimension. Then, for any real vector $b$ and any real number $a$,
\begin{equation}\label{eq:lmain}
 \mathbb{P} \left\{ \frac{\|Q(g+b)\|}{|(g+b)^tQ(g+b)-a|} \geq t \right\} \leq \frac{C}{t}, \quad t\geq1,  
 \end{equation}
for some absolute constant $C$.
\end{lemma}

The estimate \eqref{eq:vec} follows in the GOE case, by combining  \eqref{eq:key} with  Lemma~\ref{lemma:main} (through conditioning on $g_0$ and $Q$).

\medskip

\noindent {\bf The complex (GUE) case.}  Here  \eqref{eq:A + V} is replaced by
\begin{equation}\label{eq:A + V'}
A + V  = \frac{1}{\sqrt{N}} \left(  \begin{array}{cc}
g_0 +a    & (g+b) ^*\\
g+b & W + D
\end{array}\right),
\end{equation}
where $g_0,a\in\R$, $g, b\in\C^{N-1}$ and $W$ and $D$ are $(N-1)\times(N-1)$
Hermitian matrices, $g_0, g$ and $W$ are independent, $a,b$ and $D$ are deterministic, $g_0$ is a standard real Gaussian and $g$ is a standard complex Gaussian vector (i.e., with independent entries having independent real and imaginary parts, each of which has the normal distribution with mean $0$ and variance $1/2$). Following the same steps as in the GOE case one arrives at
\begin{equation}\label{eq:GUE_case_fixed_vector_expression}
\begin{split} \frac{1}{\sqrt{N}} \| (A + V)^{-1} e_1 \|
&= \frac{\sqrt{1 + \| Q (g+b)\|^2}}{|g_0 + a - (g+b)^* Q (g+b)|}\leq \frac{1 + \| Q (g+b)\|}{|g_0 + a - (g+b)^* Q (g+b)|}
\end{split}
\end{equation}
where $Q:=(W + D)^{-1}$.\\

To conclude the proof via
the arguments used in the GOE case,
we rewrite the right-hand side of \eqref{eq:GUE_case_fixed_vector_expression}
in terms of a similar expression involving only real quantities.
For this purpose we consider $\mathbb{C}^N$ with the
standard basis $(e_j)_{j=1}^N$ as a vector
space over $\mathbb{R}$ with the basis
\[ (e_1, ie_1, e_2, ie_2, \cdots, e_N, ie_N)~, \]
 and denote by $\tilde{Q}$ be  the $2N\times 2N$ real symmetric matrix which represents
multiplication by $Q$ in this basis.
For a vector $v\in\C^N$, denote by $\tilde{v}\in\R^{2N}$
its image under this identification. Then
\begin{equation}
  \|Q(g+b)\| = \|\tilde{Q}(\tilde{g}+\tilde{b})\|
\end{equation}
and, using that $(g+b)^*Q(g+b)$ is real as $Q$ is Hermitian, that
\begin{equation}
  (g+b)^*Q(g+b) = (\tilde{g}+\tilde{b})^t\tilde{Q}(\tilde{g}+\tilde{b}).
\end{equation}
Thus
\begin{equation}\label{eq:complex_to_real_1}
  \frac{1 + \| Q (g+b)\|}{|g_0 + a - (g+b)^* Q (g+b)|} = \frac{1 + \| \tilde{Q} (\tilde{g}+\tilde{b})\|}{|g_0 + a - (\tilde{g}+\tilde{b})^t \tilde{Q} (\tilde{g}+\tilde{b})|}.
\end{equation}
Note that $\tilde{g}$ is a real Gaussian vector whose entries are independent with variance $1/2$. In order to work with standard real Gaussian vectors we rewrite this expression
as
\begin{equation}\label{eq:complex_to_real_2}
  \frac{1 + \| \tilde{Q} (\tilde{g}+\tilde{b})\|}{|g_0 + a - (\tilde{g}+\tilde{b})^t \tilde{Q} (\tilde{g}+\tilde{b})|} = \frac{\sqrt{2}\big(1 + \| \frac{\tilde{Q}}{\sqrt{2}} (\sqrt{2}\tilde{g}+\sqrt{2}\tilde{b})\|\big)}{|\sqrt{2}g_0 + \sqrt{2}a - (\sqrt{2}\tilde{g}+\sqrt{2}\tilde{b})^t \frac{\tilde{Q}}{\sqrt{2}} (\sqrt{2}\tilde{g}+\sqrt{2}\tilde{b})|}~,
\end{equation}
where $\sqrt{2}\tilde{g}$ is standard Gaussian.
Using \eqref{eq:complex_to_real_1} and \eqref{eq:complex_to_real_2} with \eqref{eq:GUE_case_fixed_vector_expression} allows to finish the proof in the GUE case with the same argument as in the GOE case.\qed

\begin{rmk} Note that we actually proved the following stronger, conditional version of (\ref{eq:vec}): for $\varphi = e_1$,
the estimate (\ref{eq:vec}) holds conditionally on the sub-matrix
obtained by deleting the first row and column of $V$. For a general
$\varphi$, this translates to the following estimate, which will be
of use in the sequel:
 \begin{equation}\label{eq:fixed_vector_bound_with_conditioning}
    \mathbb{P}\left\{\|H^{-1}\vp\|\ge t\sqrt{N}\|\vp\|\,\Big|\, \left\{ u^*Hv \, \mid \, u, v \perp \varphi\right\} \right\} \le \frac{C}{t},
  \end{equation}
with a constant $C$ which is uniform in $A$, $N$ and $\vp$.
\end{rmk}

\section{Frobenius norm bound} \label{s:frob}
To  deduce the Frobenius norm estimate \eqref{eq:frob}
 from   \eqref{eq:vec}, we employ the following principle. A similar strategy
 was employed by Sankar, Spielman, and Teng \cite[Proof of Theorem 3.3]{SST}
 \begin{lemma}\label{l:qf1}
Let $Q$ be an $N\times N$ real symmetric matrix and  $\vp$ be a random
vector uniformly distributed on the sphere $\mathbb{S}^{N-1} =
\{\psi\in\R^N\colon\|\psi\|=1\}$. Then
\begin{equation*}
\mathbb{P} \left\{ \|Q \vp\| \leq \frac{\epsilon}{\sqrt{N}}\|Q\|_{\F}
\right\} \leq 5 \epsilon, \quad \epsilon>0.
\end{equation*}
\end{lemma}

\begin{proof}
By the Chebyshev inequality, for any real $\xi$,
\begin{equation}\label{eq:Q_applied_to_u_estimate}
  \mathbb{P} \left\{ \|Q \vp\| \leq \frac{\epsilon}{\sqrt{N}}\|Q\|_{\F}
\right\} \le
\exp\left(\frac{\xi\epsilon^2}{N}\|Q\|_{\F}^2\right)\mathbb{E}
\exp\left(-\xi\|Q\vp\|^2\right)  \, .
\end{equation}
A uniformly distributed vector on $\mathbb{S}^{N-1} $ can be generated by letting $\vp = g /\|g\|$ with $g$ a standard real Gaussian vector, for which
$\frac{g}{\|g\|}$ and $\|g\|$ are independent.  Thus,
\begin{equation}\begin{split}\label{eq:uniform_and_Gaussian_Laplace_transform}
\mathbb{E} \left[
\exp\left(-\xi\|Q\vp\|^2\right)  \right] & =    \mathbb{E} \left[ \exp\left(-\xi\|Q g\|^2/ \|g\|^2 \right)  \right]  \\
& = \frac{1}{\mathbb{P}\left\{\|g\|^2\le 2N\right\}}\mathbb{E} \left[\exp\left(-\xi\|Q g\|^2/ \|g\|^2\right)\mathbbm{1}_{\|g\|^2\le 2N}\right]
\\ & \le \  2\  \mathbb{E}\exp\left[-\frac{\xi}{2N}\|Qg\|^2\right],
  \end{split}\end{equation}
where use was made of the bound $\mathbb{P}(\|g\|^2\le 2N)\ge \frac{1}{2}$
which follows from $\mathbb{E}\|g\|^2 = N$.

Let $\{ {\mathcal E}_j\} $ be the eigenvalues of $Q$, with which
 $\|Q\|_{\F}^2 = \sum
{\mathcal E}_j^2$.  As the distribution of $g$ is invariant under orthogonal
transformations, and the eigenvectors of $Q$ form an orthonormal
basis, one gets (using a known Gaussian integral) for any
$\xi \geq 0$,
\begin{equation*}
  \mathbb{E}\exp\left(-\frac{\xi}{2N}\|Qg\|^2\right) = \mathbb{E}\exp\left(-\frac{\xi}{2N}\sum_{j=1}^N {\mathcal E}_j^2 g_j^2\right) =
  \prod_{j=1}^N \frac{1}{\sqrt{1+\frac{\xi}{N}{\mathcal E}_j^2}}\le \frac{1}{\sqrt{1+\frac{\xi}{N}\|Q\|_{\F}^2}}.
\end{equation*}
Juxtaposing the last inequality with
\eqref{eq:Q_applied_to_u_estimate} and
\eqref{eq:uniform_and_Gaussian_Laplace_transform}, and substituting
$\xi = \frac{N}{2\epsilon^2\|Q\|_{\F}^2}(1 - 2\epsilon^2)$, yields
\begin{equation*}
  \mathbb{P} \left\{ \|Q \vp\| \leq \frac{\epsilon}{\sqrt{N}}\|Q\|_{\F}
\right\} \le
\frac{2\exp\left(\frac{\xi\epsilon^2}{N}\|Q\|_{\F}^2\right)}{\sqrt{1+\frac{\xi}{N}\|Q\|_{\F}^2}}
= 2\sqrt{2}\epsilon \exp\left(\frac{1 - 2\epsilon^2}{2}\right)\le
5\epsilon.\qedhere
\end{equation*}
\end{proof}

We proceed to prove the Frobenius norm estimate \eqref{eq:frob} in the GOE case.
Let $\vp$ be a  random vector distributed uniformly on the sphere $\mathbb{S}^{N-1}$ and independent of $H$, and let
$t\geq1$.    Applying Lemma~\ref{l:qf1} with $Q = H^{-1}$
and $\epsilon=\frac{1}{10}$, we get
\begin{multline}\label{eq:from_Frobenius_to_fixed_vector}
\mathbb{P}\left\{\|H^{-1}\|_{\F}\ge  tN\right\}
  =  \mathbb{E} \left[\mathbbm{1}_{\|H^{-1}\|_{\F} \ge tN}\right]
  \leq \ 2\  \mathbb{E} \left[\mathbbm{1}_{\|H^{-1}\|_{\F}\ge tN}\mathbb{P}\left\{\|H^{-1} \vp\|\ge \frac{t\sqrt{N}}{10}\,\Big|\,
  H\right\}\right] \\
   \leq  \ 2\    \mathbb{E} \left[ \mathbb{P}\left\{\|H^{-1} \vp\|\ge \frac{t\sqrt{N}}{10}\,\Big|\,
 H\right\}\right] \ \leq \ 2\   \mathbb{P}  \left\{\|H^{-1} \vp\|\ge \frac{t\sqrt{N}}{10}\right\}  \,.
\end{multline}
Applying now the fixed vector bound  \eqref{eq:vec} conditionally on $\vp$  to the  probability in the last  term one gets
\be
\mathbb{P}\left\{\|H^{-1}\|_{\F}\ge  tN\right\}  \ \leq \ \frac{20\, C}{t} \, ,
\ee
i.e., \eqref{eq:frob} holds in the real  (GOE) case. \\

A similar  argument may be used to establish \eqref{eq:frob} in the
GUE case using the following complex analog to Lemma~\ref{l:qf1}. If
$Q$ is an $N\times N$ Hermitian matrix and $\vp $ is a random vector
uniformly distributed on the complex sphere,
\begin{equation}\label{eq:complex_unit_sphere}
\mathbb{S}^{N-1}_{\mathbb{C}} = \{\psi \in\C^N\colon\|\psi \|=1\},
\end{equation}
then, for all $\epsilon>0$,
\begin{equation}\label{eq:complex_version_of_uniform_vector_lemma}
  \mathbb{P} \left\{ \|Q \vp\| \leq \frac{\epsilon}{\sqrt{N}}\|Q\|_{\F}
\right\} \leq 5 \epsilon.
\end{equation}
The  inequality follows from Lemma~\ref{l:qf1} applied with $2N$ in place of $N$ by
identifying the space $\C^N$ with $\R^{2N}$ as in the proof of the GUE case of (\ref{eq:vec}).  This identification
multiplies the Frobenius norm by $\sqrt{2}$.

\section{Bound on the density of states}\label{s:dos}

We now turn to the density of states bound~\eqref{eq:dos}.

Let $H$ be the random matrix of Theorem~\ref{thm:main}.  Observe that  almost surely $H$ has only simple eigenvalues, e.g., as its distribution is absolutely continuous with respect to that of the underlying invariant Gaussian ensemble (GOE or GUE) and these are well known to have this property~\cite{Mehta}.

For a finite interval $I$, let $\{ I_{j,M}\}$ be a nested sequence of partitions of $I$ into   subintervals whose maximal length tends to zero as $M\to \infty$.   Using the simplicity of the spectrum, almost surely:
\be \label{eq:number_of_eigenvalues_formula}
\sum_{j}  \mathbbm{1} \{ \text{$H$ has an eigenvalue in $I_{j,M}\cap I$}\}
\  \mathop {\nearrow}_{M\to \infty}   \  \#\{\text{eigenvalues of $H$ in $I$}\}  \,.
\ee
Taking the expectation value and applying the monotone convergence theorem
gives
\begin{equation}
\mathbb{E}\left[\#\{\text{eigenvalues of $H$ in $I$}\}\right] = \lim_{M\to\infty} \sum_{j}  \mathbb{P} \left \{ \text{$H$ has an eigenvalue in $I_{j,M}\cap I$} \right\} .
\end{equation}
  The probabilities on the right may be estimated through the norm bound \eqref{eq:frob}, which implies that for any interval $J= [{\mathcal E} - \varepsilon, {\mathcal E} + \varepsilon ]$
 \begin{equation}
 \mathbb{P} \left \{  \text{$H$ has an eigenvalue in $J$} \right\}   =
 \mathbb{P} \left \{  \| (H-{\mathcal E})^{-1}\|_{\op} \geq \frac{2}{ |J| } \right\} \leq  C N |J| / 2 \,.
 \end{equation}
Upon summation this yields the claimed density of states bound \eqref{eq:dos}.

\section{Minami-type bound}\label{sec:min}

The proof of Theorem~\ref{thm:Minami_like} proceeds by induction on $k$, using an idea of
Combes, Germinet, and Klein \cite{CGK}. The case $k=1$ is exactly the Wegner-type estimate
\eqref{eq:dos}. Thus we assume that (\ref{eq:Minami_expectation}) is valid for a certain $k$
and prove that it is also valid for $k+1$.

In the proof we use the inequality (\ref{eq:fixed_vector_bound_with_conditioning}), which we restate
for convenience. Letting $H = A + V$ be as in Theorem~\ref{thm:main}, $\vp\in\R^N\setminus\{0\}$, denote by $H_{\vp}$ the matrix obtained by restricting $H$
  to the subspace orthogonal to $\vp$, i.e.\ the $(N-1) \times (N-1)$ matrix
  $H_\varphi = P_{\varphi^\perp} H P_{\varphi^\perp}^*$, where $P_{\varphi^\perp}$ is the orthogonal
  projection onto the orthogonal complement of $\varphi$.  Then (\ref{eq:fixed_vector_bound_with_conditioning})
  asserts that
   \begin{equation}\label{eq:fixed_vector_bound_with_conditioning'}
    \mathbb{P}\left\{\|H^{-1}\vp\|\ge t\sqrt{N}\|\vp\|\,\Big|\, H_{\vp}\right\} \le \frac{C}{t},
  \end{equation}
  with a constant $C$ which is uniform in $A$, $N$ and $\vp$.

  Let $H$ be as in Theorem~\ref{thm:main} and fix $I_1$ to be a finite
  interval. Let $\vp$ be a random vector, independent of $H$, which is uniformly distributed on the unit sphere $\mathbb{S}^{N-1}$ in the real case
  or uniformly distributed on the complex unit sphere $\mathbb{S}^{N-1}_{\mathbb{C}}$ (see \eqref{eq:complex_unit_sphere}) in the complex
  case. Lemma~\ref{l:qf1} in the real case or its complex
  version \eqref{eq:complex_version_of_uniform_vector_lemma} in the complex case,
  imply that for every non-negative random variable $X$, measurable
  with respect to $H$, and every $\mathcal E\in\R$ one has
  \begin{equation}\label{eq:expectation_via_uniform_vector}
  \begin{split}
    \mathbb{E}[X] &\le 2\mathbb{E}\left[X\cdot \mathbb{P}\left\{\|(H - \mathcal E)^{-1}\vp\|\ge \frac{\|(H -
    \mathcal E)^{-1}\|_{\F}}{10\sqrt{N}}\,\Big|\, H\right\}\right]\\
    &= 2\mathbb{E}\left[X\cdot \mathbbm{1}_{\|(H - \mathcal E)^{-1}\vp\|\ge \frac{\|(H -
    \mathcal E)^{-1}\|_{\F}}{10\sqrt{N}}}\right].
  \end{split}
  \end{equation}
  Now, let $\{ I_{j,M}\}$ be a nested sequence of partitions of $I_1$ into subintervals whose maximal length tends to zero as $M\to
  \infty$. As in Section~\ref{s:dos}, the monotone convergence theorem implies that
  \begin{multline}\label{eq:monotone_convergence_reduction}
    \mathbb{E}\left[\mathcal{N} (I_1)(\mathcal{N} (I_2)-1)_+\cdots(\mathcal{N}
    (I_{k+1})-k)_+\right]\\
     = \lim_{M\to\infty} \sum_j \mathbb{E}\left[\mathbbm{1}_{\mathcal{N} (I_{j,M}\cap I_1)\ge 1}\,(\mathcal{N} (I_2)-1)_+\cdots(\mathcal{N}
    (I_{k+1})-k)_+\right].
  \end{multline}
  We focus on estimating a single summand in the last expression.
  Let $J\subseteq I_1$ be an interval with midpoint $\mathcal E$.
The event that $\mathcal{N} (J)\ge 1$ coincides with $\|(H-\mathcal E)^{-1}\|_{\op}\ge \frac{2}{|J|}$. Applying
  \eqref{eq:expectation_via_uniform_vector},
  \begin{multline}\label{eq:adding_phi}
     \mathbb{E}\left[\mathbbm{1}_{\mathcal{N} (J)\ge 1}(\mathcal{N} (I_2)-1)_+\cdots(\mathcal{N}
    (I_{k+1})-k)_+\right]\\
     \le 2\mathbb{E}\left[\mathbbm{1}_{\|(H-\mathcal{E})^{-1}\|_{\op} \ge \frac{2}{|J|}}(\mathcal{N} (I_2)-1)_+\cdots(\mathcal{N}(I_{k+1})-k)_+\mathbbm{1}_{\|(H - \mathcal{E})^{-1}\vp\|\ge \frac{\|(H -
    \mathcal E)^{-1}\|_{\F}}{10\sqrt{N}}}\right]\\
     \le 2\mathbb{E}\left[\mathbbm{1}_{\|(H-\mathcal E)^{-1}\vp\| \ge \frac{1}{5|J|\sqrt{N}}}(\mathcal{N} (I_2)-1)_+\cdots(\mathcal{N}
    (I_{k+1})-k)_+\right].
  \end{multline}
Let $H_{\vp}$ be as above, then the eigenvalues of $H_\vp$ interlace those of $H$, therefore
$\mathcal{N} (I_j)-1\le \mathcal{N}(I_j;H_\vp)$. Thus,
  \begin{equation}\label{eq:conditioned_on_phi_and_H_phi}
  \begin{split}
    &\mathbb{E}\left[\mathbbm{1}_{\|(H-\mathcal E)^{-1}\vp\| \ge \frac{1}{5|J|\sqrt{N}}}(\mathcal{N} (I_2)-1)_+\cdots(\mathcal{N}
    (I_{k+1})-k)_+\right]\\
    &\,\,\le \mathbb{E}\left[\mathbbm{1}_{\|(H-\mathcal E)^{-1}\vp\| \ge \frac{1}{5|J|\sqrt{N}}}\,\mathcal{N} (I_2;H_{\vp})\cdots(\mathcal{N}
    (I_{k+1};H_{\vp})-k+1)_+\right]\\
    &\,\,= \mathbb{E}\left[\mathcal{N} (I_2;H_{\vp})\cdots(\mathcal{N}
    (I_{k+1};H_{\vp})-k+1)_+\,\mathbb{P}\left\{\|(H-\mathcal E)^{-1}\vp\| \ge \frac{1}{5|J|\sqrt{N}}\,\Big|\,\vp,
    H_{\vp}\right\}\right]\\
    &\,\,\le 5C|J|N\cdot \mathbb{E}\left[\mathcal{N} (I_2;H_{\vp})\cdots(\mathcal{N}
    (I_{k+1};H_{\vp})-k+1)_+\right]~,
  \end{split}
  \end{equation}
  where in the last inequality we have applied the estimate
  \eqref{eq:fixed_vector_bound_with_conditioning'} to the matrix
  $H-\mathcal E$. By the invariance of the underlying Gaussian
  ensemble (GOE or GUE), the $(N-1)$-dimensional matrix
\[ \widetilde{H_\vp} = \sqrt{\frac{N}{N-1}} H_\vp~, \] conditioned on $\vp$, has the form treated in Theorem~\ref{thm:main}. Thus the estimate
  \eqref{eq:Minami_expectation}, applied using the induction hypothesis to $\widetilde{H_\vp}$, shows that
  \begin{equation}\begin{split}\label{eq:Minami_N-1}
    \mathbb{E}\left[\mathcal{N} (I_2;H_{\vp})\cdots(\mathcal{N}
    (I_{k+1};H_{\vp})-k+1)_+\right] &\le \prod_{j=2}^{k+1} \left(C_0 |I_j| \sqrt{N(N-1)}\right) \\&\leq \prod_{j=2}^{k+1}(C_0|I_j|N)~.
\end{split}
\end{equation}
  Putting together \eqref{eq:adding_phi}, \eqref{eq:conditioned_on_phi_and_H_phi} and
  \eqref{eq:Minami_N-1} shows that
  \begin{equation*}
    \mathbb{E}\left[\mathbbm{1}_{\mathcal{N} (J)\ge 1}(\mathcal{N} (I_2)-1)_+\cdots(\mathcal{N}
    (I_{k+1})-k)_+\right]\le 10C|J|N\times \prod_{j=2}^{k+1} (C_0 |I_j| N)~.
  \end{equation*}
  Taking $C_0\ge 10C$, the theorem follows by plugging the last estimate back into
  \eqref{eq:monotone_convergence_reduction} and performing the
  summation.

\section{Ratio of quadratic forms} \label{sec:tech}

Let us recall from Section~\ref{lem:fixed_vec} that the above results hinge on the estimate stated in Lemma~\ref{lemma:main}.
The statement to be proved is that for any  (non-random) non-zero real symmetric matrix $Q$, real vector $b$, real number $a$ and $t\ge 1$,
\be  \label{eq:ratio}
\mathbb{P} \left\{ \frac{\|Q(g+b)\|}{|(g+b)^tQ(g+b)-a|} \geq t \right\} \leq \frac{C}{t},
\ee
where $g$ is a standard real Gaussian vector and $C$ is an absolute constant. \\

That such a bound may hold may be surmised from the observation that
\[ \mathbb{E}  \|Q(g+b)\|^2 \leq C \, \operatorname{Var}
[(g+b)^tQ(g+b)-a] \]
(uniformly in $Q$, $b$, and $a$), which implies that   the denominator  of the ratio in \eqref{eq:ratio} fluctuates on a scale which is not smaller than the typical size of the numerator.    However, more careful analysis is needed to take into account the   dependence of the two terms and the possibility that the denominator has unbounded probability density at  small values.

\medskip

We turn to the proof of  Lemma~\ref{lemma:main}, starting with two preliminary claims. The first covers its rank one case.
\begin{claim}\label{claim1}
If $h$ is a standard Gaussian variable, $a,b \in \mathbb{R}$, then
\begin{equation*}
\mathbb{P}\left \{ \frac{|h + b|}{|(h + b)^2 - a|} \geq t \right\} \leq \sqrt{\frac{8}{\pi}} \frac1t,\quad t\geq1.
\end{equation*}
\end{claim}
\begin{proof} The event $\frac{|h + b|}{|(h + b)^2 - a|} \geq t$ coincides with
\begin{equation}\label{eq:gaussian}
\frac{|h + b|}{t} \geq |(h + b)^2 - a|.
\end{equation}
If $a < 0$, the probability of this event will only increase if we replace $a$ with $0$, thus we suppose that $a \geq 0$. Then
\begin{equation*}
|(h + b)^2 - a| = \big||h + b| - \sqrt{a}\big|\cdot\big||h + b| + \sqrt{a}\big| \geq \big||h + b| - \sqrt{a}\big|\cdot|h + b|,
\end{equation*}
whence
\begin{equation*}
\begin{split}
\mathbb{P}\left \{\frac{|h + b|}{|(h + b)^2 - a|} \geq t\right\}&\le \mathbb{P}\left \{\frac{|h + b|}{t} \geq \big||h + b| - \sqrt{a}\big|\cdot|h + b|\right\} \\
&= \mathbb{P}\left \{\big||h + b| - \sqrt{a}\big|\le \frac{1}{t}\right\} \le \frac{4}{\sqrt{2\pi}}\frac{1}{t}.\qedhere
\end{split}
\end{equation*}
\end{proof}
The next claim will be used in deriving probability  bounds on ratios through estimates on the Fourier transform of the joint probability distribution of the numerator and denominator (also known as the joint characteristic function).

\begin{claim}\label{claim2} Let $X > 0$, $Y$ be a pair of random variables, and
\be
\chi(\xi,\eta)  := \mathbb{E} \exp{(i(\xi X + \eta Y))} \, .
\ee
Then, for any $\epsilon > 0$ and $a \in \mathbb{R}$,
\begin{equation}\label{estimate_lemma3}
\mathbb{P}\left\{ \frac{\sqrt{X}}{|Y - a|}\geq \epsilon^{-1}\right\}
\leq \frac{e^{1/4}\epsilon }{4\pi} \liminf_{\delta \to +0} \int
d\eta \left| \int { d\xi \frac{\chi(\xi,\eta)}{(\eta^2\epsilon^2 +
i\xi + \delta)^{\frac{3}{2}+\delta}} } \right|.
\end{equation}
\end{claim}

\begin{proof} The right-hand side of (\ref{estimate_lemma3}) does not change if we replace $Y$ with $Y-a$, therefore we can assume that $a=0$. Set
\[
h(x,y) = \exp\left( -\frac{y^2}{4\epsilon^2 x}\right)
\mathbbm{1}_{x> 0}, \quad h_\delta(x, y) = h(x, y) \exp(-\delta x)
x^\delta,
\]
and note that
\[ h(x,y) \geq e^{-\frac{1}{4}}\mathbbm{1}_{\frac{\sqrt{x}}{|y|}\geq\epsilon^{-1}}\mathbbm{1}_{x> 0}.\]
Therefore by the Chebyshev inequality and the Fatou lemma,
\[
\mathbb{P}\left\{ \frac{\sqrt{X}}{|Y|}\geq \epsilon^{-1}\right\}
\leq e^{\frac{1}{4}}\mathbb{E}h(X, Y) \leq e^{\frac{1}{4}}\liminf_{\delta
\to +0} \mathbb{E}h_\delta(X, Y).\]
The function $h_\delta$ is continuous and
 integrable, and its Fourier transform $\hat{h}_\delta$ is
also integrable, as follows from the explicit computation below.
Therefore, by a version of the Plancherel theorem for the
Fourier--Stieltjes transform \cite[\S VI.2]{Katz},
\[ \mathbb{E}h_\delta(X, Y) = \left(\frac{1}{2\pi}\right)^2\iint d\xi d\eta\,   \widehat{h}_\delta (\xi, \eta)\chi(\xi, \eta),
\]
where
\[ \widehat{h}_\delta(\xi, \eta) = \iint h_\delta(x, y)\exp(-i(\xi x + \eta y)) dx dy.\]
To compute $\widehat{h}_\delta$  we first fix $x > 0$ and integrate over $y$ (using a standard Gaussian integral)
\[
\int_{-\infty}^\infty h(x,y)\, \exp(-i\eta y) dy =
\int_{-\infty}^\infty \exp\left[-\frac{y^2}{4\epsilon^2 x} - i\eta
y\right] dy = 2\sqrt{\pi x}\, \epsilon \, \exp(-\eta^2\epsilon^2 x).
\]
Multiplying by $e^{-\delta x} x^\delta$ and integrating over $x$,
\[
\widehat{h}_\delta(\xi, \eta) = 2\sqrt{\pi}\epsilon \int_{0}^\infty x^{\frac12 + \delta} \exp(-x(\eta^2\epsilon^2 + i\xi + \delta)) dx \  =\
\frac{2\sqrt{\pi}\Gamma(\frac{3}{2} + \delta)\epsilon }{(\eta^2\epsilon^2 + i\xi + \delta)^{\frac{3}{2}+\delta}}~.
\]
This implies
\[ \left(\frac{1}{2\pi}\right)^2\iint d\xi d\eta\,   \widehat{h}_\delta (\xi, \eta)\chi(\xi, \eta) =
\frac{\Gamma(\frac{3}{2}+\delta) \epsilon}{2\pi^{3/2}} \iint d\xi
d\eta\, (\eta^2\epsilon^2 + i\xi + \delta)^{-\frac{3}{2}-\delta} \,
\chi(\xi, \eta).\] Applying the   Fubini theorem and taking absolute
value, we finally obtain:
\[\mathbb{P}\left\{ \frac{\sqrt{X}}{|Y|}\geq \epsilon^{-1}\right\} \leq \frac{e^{\frac{1}{4}}\epsilon}{4\pi} \liminf_{\delta \to +0}  \int d\eta\, \left| \int d\xi\, \frac{\chi(\xi,\eta)}{(\eta^2\epsilon^2 + i\xi + \delta)^{\frac{3}{2}+\delta}}\right|.\qedhere
\]
\end{proof}

\begin{proof}[Proof of Lemma~\ref{lemma:main}]
Using the symmetry which is built into the assumptions, it suffices to establish the bound for diagonal matrices $Q = \mathrm{diag}({\mathcal E}_1, {\mathcal E}_2, \cdots)$. Our goal is to prove that
\begin{equation}\label{eq:main_lemma_goal_with_lambda}
\mathbb{P} \left\{ \frac{\sqrt{\sum_{j\ge 1} {\mathcal E}_j^2(g_j +
b_j)^2}}{|\sum_{j \geq 1}{\mathcal E}_j (g_j + b_j)^2-a|} \geq t \right\}
\leq \frac{C}{t}, \quad t\geq 1,
\end{equation}
where the sums may be restricted to ${\mathcal E}_j \neq 0$ (and the probability average is over the independent standard Gaussian variables $g_j$). \\

We reorder the eigenvalues $({\mathcal E}_j)$ so that
\[ {\mathcal E}_1^2(1+b_1^2) \geq {\mathcal E}_2^2(1+b_2^2) \geq {\mathcal E}_3^2(1+b_3^2) \geq \cdots.\]

Denote
\begin{equation}\label{eq:r_def}
r := \begin{cases}
0, &{\mathcal E}_1^2(1+b_1^2) \leq \frac{1}{10}\sum_{j > 1}{\mathcal E}_j^2(1+b_j^2)\\
1, & {\mathcal E}_1^2(1+b_1^2) > \frac{1}{10}\sum_{j > 1}{\mathcal E}_j^2(1+b_j^2),\; {\mathcal E}_2^2(1+b_2^2) \leq \frac{1}{10}\sum_{j > 2}{\mathcal E}_j^2(1+b_j^2)\\
2, & \text{otherwise}
\end{cases}\end{equation}
and \be X := \sum_{j > r}{\mathcal E}_j^2(g_j + b_j)^2,\, \quad \, Y
:= \sum_{j \geq 1}{\mathcal E}_j (g_j + b_j)^2,  \, \quad \, \chi
(\xi, \eta) := \mathbb{E} \exp(i(\xi X + \eta Y))  \,, \ee where,
according to the number of non-zero eigenvalues, $X$ is either
identically zero or almost surely positive. Observe that
\begin{equation}\label{eq:twoparts}
\begin{split}
\sqrt{\sum_{j \geq 1}{\mathcal E}_j^2(g_j + b_j)^2} \leq
\sum_{j=1}^r |{\mathcal E}_j||g_j + b_j| + \sqrt{X}.
\end{split}
\end{equation}
For the terms in the first sum in the right-hand side of (\ref{eq:twoparts}), Claim~\ref{claim1} yields
\be
\mathbb{P} \left\{ \frac{|{\mathcal E}_j||g_j + b_j|}{|Y-a|} \geq t
\right\} \leq \sqrt{\frac{8}{\pi}}\frac{1}{t}, \quad t\geq 1. \ee
Thus, to prove \eqref{eq:main_lemma_goal_with_lambda} it suffices to
show that \be\label{eq:main_lemma_goal_with_r_and_lambda} \mathbb{P}
\left\{ \frac{\sqrt{X}}{|Y-a|} \geq t \right\} \leq \frac{C}{t},
\quad t\geq 1~. \ee If $X$ is identically zero the inequality is
trivial. Thus we assume that $X$ is not identically zero and note that this assumption entails that ${\mathcal E}_1, {\mathcal E}_2, {\mathcal E}_3\neq 0$. We now use Claim~\ref{claim2} which reduces the task of proving \eqref{eq:main_lemma_goal_with_r_and_lambda} to showing that
\begin{equation}\label{eq:lemma_reduced_to_integral}
 \liminf_{\delta' \to +0} \int  d\eta \left| \int { d\xi \frac{\chi(\xi,\eta)}{(\eta^2\epsilon^2 + i\xi + \delta')^{\frac{3}{2}+\delta'}} }\right|\le C\,.
\end{equation}
Noting that a standard Gaussian random variable $h$ satisfies
\begin{equation}\label{eq:Gaussian_Fourier_transform}
\mathbb{E} \exp(i\alpha (h+\beta)^2) = \frac{1}{\sqrt{1-2i\alpha}}\exp\left(\frac{i\alpha}{1-2i\alpha}\beta^2\right),
\end{equation}
we have
\be  \begin{split}
\chi (\xi, \eta)
&= \mathbb{E} \left[ \exp\left(i\bigg(\sum_{j=1}^r \eta {\mathcal E}_j (g_j + b_j)^2 + \sum_{j>r} (\xi{\mathcal E}_j^2+\eta{\mathcal E}_j)(g_j + b_j)^2\bigg)\right) \right] \\
&= \prod_{j=1}^r \frac{1}{\sqrt{1 - 2i\eta{\mathcal E}_j}}\exp\left( b_j^2 \frac{i\eta{\mathcal E}_j}{1 - 2i\eta{\mathcal E}_j}\right)\\
&\quad\times \prod_{j>r} \frac{1}{\sqrt{1 - 2i(\xi{\mathcal E}_j^2 + \eta{\mathcal E}_j)}} \exp \left( b_j^2\frac{i(\xi{\mathcal E}_j^2 + \eta{\mathcal E}_j)}{1 - 2i(\xi{\mathcal E}_j^2 + \eta{\mathcal E}_j)}\right).
\end{split}
\ee
For real $\eta$, the function $\chi(\cdot,\eta)$ has an analytic continuation to the domain
\begin{equation}\label{eq:acont}
\left\{ \xi - i \delta \, \mid \, \xi \in \mathbb{R}~, \, \delta < \frac{1}{2\max_{j>r} {\mathcal E}_{j}^2}\right\}~;
\end{equation}
this continuation is given by
 \begin{equation}\label{eq:acontformula}\begin{split}
\chi(\xi - i\delta,\eta) &= \prod_{j=1}^r \frac{1}{\sqrt{1 - 2i\eta{\mathcal E}_j}}\exp\left( b_j^2 \frac{i\eta{\mathcal E}_j}{1 - 2i\eta{\mathcal E}_j}\right)\\
&\times \prod_{j>r} \frac{1}{\sqrt{(1 - 2\delta{\mathcal E}_j^2) - 2i\zeta_j}} \exp \left( b_j^2\frac{\delta{\mathcal E}_j^2 + i\zeta_j }{(1 - 2\delta{\mathcal E}_j^2) - 2i\zeta_j}\right),
\end{split}\end{equation}
where we set
\begin{equation}\label{eq:zeta_j_def}
\zeta_j := \xi{\mathcal E}_j^2 + \eta{\mathcal E}_j.
\end{equation}
Due to the assumption that there are at least 3 non-zero eigenvalues, we have:
\begin{equation*}
  \iint  d\xi d\eta \frac{|\chi(\xi - i\delta,\eta)|}{(\xi^2 + \delta^2)^\frac{3}{4}}<\infty
 \quad \text{for} \quad 0<  \delta < \frac{1}{2\max_{j>r} {\mathcal E}_{j}^2}~.
\end{equation*}
Thus we may change the contour of integration and apply the dominated convergence theorem to obtain that
\begin{equation}\label{eq:contour_change}
\liminf_{\delta' \to +0}   \int  d\eta \left| \int { d\xi \frac{\chi(\xi,\eta)}{(\eta^2\epsilon^2 + i\xi + \delta')^{\frac{3}{2}+\delta'}} }\right|\le  \int  d\eta \int d\xi \frac{|\chi(\xi - i\delta,\eta)|}{(\xi^2 + \delta^2)^\frac{3}{4}}.
\end{equation}

We proceed to prove \eqref{eq:lemma_reduced_to_integral} by bounding the right-hand side of \eqref{eq:contour_change} for a suitable $\delta$.
Let
\begin{equation}\label{eq:nu_delta_def}
\nu^2:= \sum_{j>r} {\mathcal E}_j^2(1 + b_j^2)\quad\text{and}\quad \delta:=\frac{1}{10\nu^2}~,
\end{equation}
and observe that
\[ \delta \leq \frac{1}{10 \max_{j > r} \mathcal{E}_j^2}~. \]
Then from (\ref{eq:acontformula})
\[\begin{split}
|\chi(\xi - i\delta,\eta)| &= \prod_{j=1}^r \frac{1}{(1 + 4\eta^2{\mathcal E}_j^2)^{\frac{1}{4}}}\exp\left( -2b_j^2 \frac{\eta^2{\mathcal E}_j^2}{1 + 4\eta^2{\mathcal E}_j^2}\right)\\
&\times \prod_{j>r} \frac{1}{((1 - 2\delta{\mathcal E}_j^2)^2 + 4\zeta_j^2)^{\frac{1}{4}}} \exp \left( b_j^2\frac{\delta{\mathcal E}_j^2(1 - 2\delta{\mathcal E}_j^2) - 2\zeta_j^2 }{(1 - 2\delta{\mathcal E}_j^2)^2 + 4\zeta_j^2}\right).
\end{split}\]
Note that, for our choice \eqref{eq:nu_delta_def} of $\nu$ and $\delta$,
\begin{equation*}
  \frac{1}{((1 - 2\delta{\mathcal E}_j^2)^2 + 4\zeta_j^2)} \le \frac{1}{(1 - 2\delta{\mathcal E}_j^2)^2(1 + 4\zeta_j^2)}
  \leq \frac{\exp\left\{ 10 \delta \mathcal{E}_j^2 \right\}}{1 + 4\zeta_j^2}
\end{equation*}
and
\begin{equation*}
  \exp \left( b_j^2\frac{\delta{\mathcal E}_j^2(1 - 2\delta{\mathcal E}_j^2)}{(1 - 2\delta{\mathcal E}_j^2)^2 + 4\zeta_j^2}\right)\le \exp\left(\frac{\delta b_j^2{\mathcal E}_j^2}{1-2\delta{\mathcal E}_j^2}\right)\le \exp\left(10 \delta b_j^2{\mathcal E}_j^2\right)~;
\end{equation*}
consequently,
\begin{equation*}
  |\chi(\xi - i\delta,\eta)|\le e \prod_{j=1}^r \frac{1}{(1 + 4\eta^2{\mathcal E}_j^2)^{\frac{1}{4}}} \exp\left( -\frac{2 \eta^2 b_j^2 {\mathcal E}_j^2}{1 + 4\eta^2{\mathcal E}_j^2}\right)\prod_{j>r} \frac{1}{(1 + 4\zeta_j^2)^{\frac{1}{4}}}\exp \left(-\frac{2b_j^2\zeta_j^2}{1 + 4\zeta_j^2} \right)~.
\end{equation*}
Combining this bound with H\"older's inequality yields
\begin{equation}\label{eq:Fourier_transform_bound}\begin{split}
&\int  d\xi \int d\eta \frac{|\chi(\xi - i\delta,\eta)|}{(\xi^2 + \delta^2)^\frac{3}{4}}\\
 &\quad\leq e\prod_{j=1}^r\left(\iint \frac{d\xi}{{(\xi^2+\delta^2)^\frac{3}{4}}} \frac{d\eta}{(1+4\eta^2{\mathcal E}_j^2)^\frac{3}{4}}
 \exp\left( -\frac{6\eta^2 b_j^2 {\mathcal E}_j^2}{1 + 4\eta^2{\mathcal E}_j^2}\right) \right)^\frac{1}{3}\\
&\quad\times \left( \iint \frac{d\xi d\eta}{{(\xi^2+\delta^2)^\frac{3}{4}}} \prod_{j > r}\frac{1}{(1 + 4\zeta_j^2)^{\frac{3}{4(3 - r)}}}\exp\left[-\frac{6b_j^2}{3-r}\frac{\zeta_j^2}{1+ 4\zeta_j^2}\right] \right)^\frac{3-r}{3}\\
&\quad=:e\prod_{j=1}^r (I_j)^{\frac{1}{3}} \times (I')^{\frac{3-r}{3}}.
\end{split}\end{equation}
The first $r$ integrals satisfy
\begin{equation}\label{eq:first_r_integrals_bound}
I_j=\frac{1}{|{\mathcal E}_j|}\frac{1}{\sqrt{\delta}} \int \frac{d\xi}{(1
+\xi^2)^\frac{3}{4}}
 \int
 \frac{d\eta}{(1+4\eta^2)^\frac{3}{4}}\exp\left( - \frac{6b_j^2\eta^2}{1 + 4\eta^2}\right)
= \frac{C_1}{|{\mathcal E}_j|(1+|b_j|)\sqrt{\delta}}\le C_2,
\end{equation}
for absolute constants $C_1, C_2$, where the last inequality uses
the choice \eqref{eq:r_def} of $r$ and the definition
\eqref{eq:nu_delta_def} of $\nu$ and $\delta$.

It remains to estimate $I'$. An additional application of H\"older's inequality
with exponents
\[ \alpha_j = \frac{{\mathcal E}_j^2(1 + b_j^2)}{\nu^2}\]
shows that
\begin{equation}\label{eq:I_r_plus_1_evaluation}
\begin{split}
  I' &\le \prod_{j > r}\left( \iint \frac{d\xi d\eta}{\left(\xi^2+\delta^2\right)^{\frac{3}{4}}}  \left(1 + 4\zeta_j^2\right)^{-\frac{3}{4(3 - r)\alpha_j}} \exp\left[-\frac{\zeta_j^2}{1+ 4\zeta_j^2}\cdot\frac{2b_j^2}{\alpha_j}\right] \right)^{\alpha_j}\\
   &=:\prod_{j > r}\big(I_{j}\big)^{\alpha_j}~.
\end{split}
\end{equation}
We proceed to show that each of the $I_{j}$ is bounded by an absolute constant. Recalling the
definition \eqref{eq:zeta_j_def} of $\zeta_j$ and changing variables,
\begin{equation}\label{eq:I_r_plus_1_j_evaluation}
\begin{split}
  I_{j} &= \frac{1}{|{\mathcal E}_j|\sqrt{\delta}} \int \frac{d\xi}{\left(1+\xi^2\right)^{\frac{3}{4}}} \int \frac{d\zeta_j}{\left(1 + 4\zeta_j^2\right)^{\frac{3}{4(3 - r)\alpha_j}}} \exp\left[-\frac{\zeta_j^2}{1+ 4\zeta_j^2}\cdot\frac{2b_j^2}{\alpha_j}\right] \\
   &=\frac{C_3}{|{\mathcal E}_j|\sqrt{\delta}} \int \frac{d\zeta_j}{\left(1 + 4\zeta_j^2\right)^{\frac{3}{4(3 - r)\alpha_j}}}
   \exp\left[-\frac{\zeta_j^2}{1+ 4\zeta_j^2}\cdot\frac{2b_j^2}{\alpha_j}\right]
\end{split}
\end{equation}
for an absolute constant $C_3 > 0$. By the choice \eqref{eq:r_def}
of $r$ and \eqref{eq:nu_delta_def} of $\nu$ and $\delta$,
\begin{equation*}
  \frac{3}{4(3 - r)\alpha_j}= \frac{3\nu^2}{4(3 - r){\mathcal E}_j^2(1+b_j^2)} \ge \frac{3}{4}
  \quad \text{for all $j>r$}~,
\end{equation*}
whence, splitting the domain of integration into $|\zeta_j|<1$ and $|\zeta_j| \geq 1$,
\begin{equation*}
  \int \frac{d\zeta_j}{\left(1 + 4\zeta_j^2\right)^{\frac{3}{4(3 - r)\alpha_j}}}
   \exp\left[-\frac{\zeta_j^2}{1+ 4\zeta_j^2}\cdot\frac{2b_j^2}{\alpha_j}\right]
   \leq \frac{C_4 \sqrt{\alpha_j}}{\max(1, |b_j|)}
   \le \frac{C_5|{\mathcal E}_j|}{\nu}
\end{equation*}
for absolute constants $C_4,C_5>0$. Plugging the result into
\eqref{eq:I_r_plus_1_j_evaluation} and then into
\eqref{eq:I_r_plus_1_evaluation} shows that $I'$ is bounded by an
absolute constant. Combining with the bounds
\eqref{eq:first_r_integrals_bound} and plugging into
\eqref{eq:Fourier_transform_bound} and \eqref{eq:contour_change}, we
conclude that
 \eqref{eq:lemma_reduced_to_integral}
 holds, and therefore so does Lemma~\ref{lemma:main}.
\end{proof}

\section{Discussion} \label{sec:disc}
\noindent{\bf Sharpness of the estimates}
The key step in our discussion of the invertibility properties of $A+V$, for a fixed Hermitian, real or complex, matrix $A$,  and  a random perturbation  $V$ sampled from a corresponding Gaussian random matrix ensemble,  was  the fixed vector bound \eqref{eq:vec}.
It may be of interest  to note that up to multiplicative constant \eqref{eq:vec}  is saturated in two very different situations:
\begin{enumerate}
\item  $A\  =\ 0 $ (or slightly more generally $ A=\mathcal E \, \mathbbm{1}$, with $|\mathcal E|<2$).   In this
case,
$\| H^{-1} \vp\| $ is typically of the order of the contribution of the closest eigenfunction, and for that, typically:
\be
\text{dist} (0, \text{spec} (H))  
\asymp 1/N \, \quad  \mbox{and} \qquad  |(\vp, \Psi_1)| \asymp
1/\sqrt{N} \,, \ee where $\Psi_1$ is the eigenfunction of eigenvalue
closest to $\mathcal E$.
\item   $A =  N^{1/2 + \varepsilon } \  P_{\vp} ^\perp $,
with $P_{\vp} $ the orthogonal projection on the space spanned by ${\vp} $ and  $P_{\vp} ^\perp$   its orthogonal complement.
Perturbation theory   allows to conclude that in this case, typically:
\be
\text{dist} (0, \text{spec} (H))
\asymp 1/\sqrt N \, \quad  \mbox{and} \qquad  |(\vp, \Psi_1)| \asymp 1 \, .
\ee
\end{enumerate}
\medskip
In both cases  $\| H^{-1} \vp\| $  is (typically) of the order of the most singular contribution,  which is
$     |(\vp, \Psi_1)|  \, \text{dist}^{-1}(0, \operatorname{spec} H) $, and hence
\be  \ \| H^{-1} \vp\|  \  \asymp \sqrt N\,  \ee
 up to a random factor whose distribution has $1/t$ tails.   However the  composition of this bound is quite different in the above two cases. \\

Note that, while in the above two cases $\| H^{-1} \vp\| $ is of the same order, the same cannot be said for the density of states at energy $0$: it scales as $N$ in the first case (i.e. up to a constant as \eqref{eq:dos}), but only as $\sqrt{N}$  in the second case.   \\

The Minami-type bound (\ref{eq:Minami_probability}) is not expected to be sharp 
since  one expects the eigenvalue repulsion to result in a higher power  on  the right-hand side of
\eqref{eq:Minami_probability} when $k\ge 2$ (namely, $k^2$ in the GUE case and $k(k+1)/2$ in the GOE case).

\noindent{\bf Weak disorder limit}    To probe the effects of weak disorder  one may consider operators of the form:
\be
H_{\lambda,N} \ = \ A_N \ +\  \lambda\,  V_N^{\text{GOE,GUE}} \, .
\ee
with $\lambda \geq 0 $  a parameter which allows to tune the strength of the disorder.
The bounds derived here share the property of the random-potential Wegner estimate, that at weak disorder the constants  degrade at the rate
$\lambda^{-1}$.   \\
\noindent{Question:}  Can the density of states bound for $H_{\lambda,N}$ be improved in case  the base operator $H_{0,N}=A_N$ is itself  asymptotically of  a bounded density of states?    \\
(The  question is open and of interest also in the original  Wegner case.)  \\

\noindent {\bf Wigner matrices} It is natural to consider extensions
of the bounds in Theorem~\ref{thm:main} to deformed Wigner matrices,
about which much has recently been learned~\cite{KY,LSSY}.  The
bounds cannot hold for any distribution of the entries:  in case
$V$ is a Wigner matrix with Bernoulli entries (uniformly sampled
from $\{ \frac{-1}{\sqrt N}, \frac{1}{\sqrt N}\}$)
and
 $\sqrt{N} A = e_1 e_1^* + M \sum_{j=2}^N e_j e_j^*$
\begin{equation*}
  \|(A + V)^{-1}\|_{\op}\to\infty\quad\text{as $M\to\infty$, on the event that $V_{11} = -\frac{1}{\sqrt{N}}$}.
\end{equation*}
In particular,  for the supremum over $N\times N$  real symmetric matrices we have  :
\begin{equation*}
  \sup_A \mathbb{P}(\|(A + V)^{-1}\|_{\op}\ge t)\ge \frac{1}{2}\quad\text{for any $t$},
\end{equation*}
in contrast to  \eqref{eq:frob}. Still, it seems reasonable to expect that bounds analogous to those presented in Theorem~\ref{thm:main} should hold when the entries of the Wigner matrix are sufficiently regular,
e.g.\ with probability densities  bounded by $\sqrt N$.\\  

\paragraph{Acknowledgment} We thank R.~Vershynin and R.~Koteck\'y for helpful suggestions.  

\smallskip
Parts of this work were done during the stays of JS and MS at
the Institute for Advanced Study in Princeton, of MA
at the Weizmann Institute of Science, Department of Mathematics and  Faculty of Physics, and of MA and SS at the Erwin Schr\"odinger International Institute for Mathematical Physics. We thank these institutions for their hospitality.

\smallskip
MA is supported in part by NSF grant PHY-1305472. RP is supported by an ISF grant and an IRG grant. JS is supported in part by The Fund For Math and NSF
grant DMS-0846325. MS is supported in part by the ISF grant 147/15. SS is supported in part by the European Research Council start-up grant 639305 (SPECTRUM).

\begin{flushright}
aizenman@princeton.edu,\\
peledron@post.tau.ac.il,\\
jeffrey@math.msu.edu,\\
mira.shamis@weizmann.ac.il,\\
sashas{\MVOne}@post.tau.ac.il
 \end{flushright}

\end{document}